\documentclass[11pt]{amsproc}
\usepackage{mathrsfs}
\usepackage{stmaryrd}
\usepackage{cases}
\usepackage{amssymb}
\usepackage{amsmath}
\usepackage{amsfonts}
\usepackage{graphicx}
\usepackage{amsmath,amstext,amsbsy,amssymb, color}
\newtheorem{theorem}{Theorem}[section]
\newtheorem{lemma}[theorem]{Lemma}
\newtheorem{proposition}[theorem]{Proposition}
\newtheorem{corollary}[theorem]{Corollary}
\theoremstyle{definition}
\newtheorem{definition}[theorem]{Definition}

\theoremstyle{remark}
\newtheorem{remark}[theorem]{Remark}

\numberwithin{equation}{section} \errorcontextlines=0

\newcommand{\Pf}{\mathrm{Pf}}

\newcommand{\cdet}{\mathrm{cdet}}
\newcommand{\rdet}{\mathrm{rdet}}

\newcommand{\ot}{\otimes}
\newcommand{\id}{\mathrm{id}}
\newcommand{\si}{\sigma}
\newcommand{\GL}{\mathrm{GL}}
\newcommand{\gl}{\mathfrak{gl}}

\begin{document}

\title[Capelli identity on multiparameter quantum linear groups]
{Capelli identity on multiparameter quantum linear groups}
\author{Naihuan Jing}
\address{NJ: Department of Mathematics, Shanghai University,
Shanghai 200444, China and Department of Mathematics, North Carolina State University, Raleigh, NC 27695, USA}
\email{jing@ncsu.edu}
\author{Jian Zhang}
\address{JZ: Institute of Mathematics and Statistics, University of Sao Paulo, Sao Paulo, Brazil 05315-970}
\email{zhang@ime.usp.br}
\thanks{{\scriptsize
\hskip -0.6 true cm MSC (2010): Primary: 17B37; Secondary: 58A17, 15A75, 15B33, 15A15.
\newline Keywords: multiparameter quantum groups, $q$-determinants, quasi-center, Capelli identity, $q$-Pfaffians, $q$-hyper-Pfaffians\\
Supported by NSFC (11531004), 
Fapesp (2015/05927-0) and Humboldt Foundation.}}
\thanks{{\scriptsize\hskip -0.6 true cm *Corresponding author}}

\begin{abstract}
A quantum Capelli identity is given on the multiparameter quantum general linear group
based on the $(p_{ij}, u)$-condition. The multiparameter quantum Pfaffian of the $(p_{ij}, u)$-quantum group is also introduced and the transformation under the congruent action is given. Generalization to the multiparameter hyper-Pfaffian and relationship with
the quantum minors are also investigated.
\end{abstract}
\maketitle
\section{Introduction}

In the study of quantum groups, multiparameter quantum groups and quantum enveloping algebras
were considered along the line of one-parameter quantum groups \cite{T, S, R, DD, DPW} and have been a
source for further developments. It was known that several of their important properties
are similar to the one-parameter analog, for example, Artin-Schelter-Tate \cite{AST} showed that the multiparameter
general linear quantum group has the same Hilbert function for the polynomial functions in $n^2$ variables under the
so-called $(p, \lambda)$-condition (see \eqref{e:cond-p} or $(p, u)$-condition in our notation).
Further results have been established for two-parameter and multi-parameter quantum groups \cite{Sc, BG, HLM, LS, JL}
such as the multiparameter quantum determinant, which converts quantum semigroups
into quantum groups.
Recently it is known that the quantum Pfaffians
can be extended to two-parameter quantum groups as well \cite{JZ2}.

In this paper, we study the generalization of the classical action of
$\GL_n(\mathbb C)\times \GL_n(\mathbb C)$ on the ring of
differential operators $\mathcal{P}(\mathrm{M}_n(\mathbb C))$, which provides
a complete set of generators for the center of $\mathrm{U}(\gl_n(\mathbb C))$. This problem
has been an inspiring example in the classical invariant theory \cite{W} and offers
one context for the celebrated Capelli identity and
Howe's dual pairs \cite{H}.  The quantum Capelli identity for Yangians was first given in \cite{N} and has played
an important role in the development of the Yangian algebra $Y(\gl_n)$ \cite{M}. The quantum analog for the quantum general linear group was
established in \cite{NUW} using Jimbo's central element in $\mathrm{U}_q(\gl_n)$.

As an interesting application of the multiparameter quantum determinant, we
formulate a generalized Capelli identity on the multiparameter general linear
group $\GL_{p, u}(n)$ by constructing a quasi-central element in the
dual algebra $\mathrm{U}_{p, u}(\gl_n)$ using the fusion method. The specialization of the Capelli identity
at equal parameters $p_{ij}=u$ matches with the quantum
Capelli identity given in \cite{NUW}.
We also generalize our recent study of quantum Pfaffians from
two-parameter quantum groups to the multiparameter case.

Using the general construction of Faddeev-Reshetikhin-Takhtajan \cite{FRT, RTF}, we construct a pair of Hopf algebras to deform
$\GL_n$ and $\mathrm{U}(\gl_n)$ for a multiparameter R-matrix. 
Our main technique is to use quadratic algebras \cite{M} to study quantum
determinants and quantum Pfaffians, and express them as the scaling constants of
quantum differential forms (cf. \cite{JZ1}).
In particular, we prove that the multiparameter quantum Pfaffian can be defined for a more general class of multiparameter quantum matrices
and derive an identity between the quantum determinant and Pfaffian. Moreover, we give
the integrality property for multiparameter quantum Pfaffian under the $(p_{ij}, u)$-conditions.

We also formulate the multiparameter quantum determinants in terms of the quasideterminant of the generating matrix.
Generalizing the one-form and two-form, we obtain higher degree analogs of the multiparameter Pfaffians and
establish the transformation rule of the multiparameter quantum hyper-Pfaffian under the quantum determinant,
which extends the fundamental transformation rule of Pfaffians in linear algebra.

The paper is organized as follows. In section two, we introduce the multiparameter quantum general linear group $\GL_{p, u}(n)$ via the
generalized quantum Yang-Baxter R-matrix and quantum determinant. We include the
details of multiparameter quantum determinant for latter generalization. 
In section three, we formulate the quasideterminant version of $\GL_{p, u}(n)$. In section four we define the dual Hopf algebra $\mathrm{U}_{p, u}(\gl_n)$ and factorize the matrix of the quantum root vectors
by that of quantum partial derivatives. Using a special quasi-central element of $\mathrm{U}_{p, u}(\gl_n)$, we derive the quantum
Capelli identity in the multiparameter setting. In section five, we introduce the notion of multiparameter quantum Pfaffian
and establish its integrality under $(p_{ij}, u)$-condition and obtain the fundamental transformation formula relating the quantum Pfaffian and determinant. Finally in section six, we study the multiparameter higher quantum Pfaffians using $q$-forms.

\section{The multiparameter quantum general linear group $\mbox{GL}_{p, u}(n)$}

Let $(p_{ij}, q_{ij})$ $(i<j)$ and $u$ be parameters satisfying the following condition \cite{AST}:
\begin{equation}\label{e:cond-p}
p_{ij}q_{ij}=u^2, \quad u^2\neq -1, \quad i<j.
\end{equation}
For convenience, we also add $p_{ij}$, $q_{ij}$ for $i\geq j$ so that
\begin{align}
p_{ij}p_{ji}&=q_{ij}q_{ji}=1,\\
p_{ii}&=q_{ii}=1.
\end{align}
The multiparameter $R$ matrix in $\mathrm{End}(\mathbb C^n\otimes \mathbb C^n)\simeq\mathrm{End}(\mathbb C^n)^{\ot 2}$ is given as:
\begin{align*}
R&=u\sum_{i}e_{ii}\otimes e_{ii}+u^{-1} \sum_{i> j}p_{ji} e_{ii}\otimes e_{jj}
 +u^{-1}\sum_{i< j}q_{ij}e_{ii}\otimes e_{jj}\\
&\qquad\qquad +(u-u^{-1})\sum_{i>j}e_{ij}\otimes e_{ji},
\end{align*}
%
where $e_{ij}$ are the standard basis elements of $\mathrm{End}(\mathbb C^n)$. The $R$ matrix
satisfies the well-known Yang-Baxter equation:
\begin{align*}
R_{12}R_{13}R_{23}=R_{23}R_{13}R_{12},
\end{align*}
where $R_{ij}\in \mathrm{End}(\mathbb C^n\otimes \mathbb C^n\otimes \mathbb C^n)$ acts on the $i$th and $j$th copies of $\mathbb C^n$ as $R$ does on $\mathbb C^n\otimes \mathbb C^n$. Note that it specializes to the two-parameter case with $p_{ij}=s, q_{ij}=r^{-1}$ and $u=\sqrt{\frac{s}{r}}$ up to an overall constant \cite{JL}. In particular, the one-parameter case corresponds to $p_{ij}=q_{ij}=u=q$ \cite{J}.

Let $\mathbb C\langle x_1, \ldots,x_n\rangle$ be the free algebra generated by $x_{i}, 1\leq i \leq n$, and $P$ the permutation operator on $\mathbb{C}^n \ot \mathbb{C}^n$ defined by $P(w\ot v)=v\ot w$, $w, v\in\mathbb C^n$.
We define the $R$-matrices $R^{\pm}$ associated with $R$ by $R^+=PRP$, $R^-=R^{-1}$. Explicitly
\begin{align}\nonumber
&R^{+}=u\sum_{i}e_{ii}\otimes e_{ii}
+u^{-1}\sum_{i< j}p_{ij}e_{ii}\otimes e_{jj}
+u^{-1} \sum_{i> j}q_{ji}e_{ii}\otimes e_{jj}\\
&\hskip 2.5in +(u-u^{-1})\sum_{i<j}e_{ij}\otimes e_{ji}, \\ \nonumber
&R^-=u^{-1}\sum_{i}e_{ii}\otimes e_{ii}+u\sum_{i> j}p_{ji}^{-1} e_{ii}\otimes e_{jj}
+u\sum_{i< j}q_{ij}^{-1} e_{ii}\otimes e_{jj}\\
&\hskip 2.5in+(u^{-1}-u )\sum_{i>j}e_{ij}\otimes e_{ji}.
\end{align}

Following the general idea of the FRT construction \cite{FRT}, we introduce $\mathrm{M}_{p, u}(n)$
as the unital associative algebra generated by $t_{ij}, 1 \leq i,j \leq n$ 
subject to the quadratic relations defined in $\mathrm{End}(\mathbb C^n\ot \mathbb C^n)\ot \mathrm{M}_{p, u}(n)$ by
\begin{align*}
R T_{1} T_{2}=T_{2}T_{1} R,
\end{align*}
where $T=(t_{ij})$, $T_1=T\ot I$, and $T_2=I\ot T$.

These relations can be expressed explicitly as follows:
\begin{align}\label{relation1}
&t_{ik}t_{il}=q_{kl}t_{il}t_{ik}, \\ \label{relation2}
&t_{ik}t_{jk}=p_{ij}t_{jk}t_{ik}, \\ \label{relation3}
&q_{kl}t_{il}t_{jk}=p_{ij}t_{jk}t_{il},\\
&q_{kl}^{-1}t_{ik}t_{jl}-q_{ij}^{-1}t_{jl}t_{ik}=(1-u^{-2})t_{il}t_{jk}, \label{relation4}
\end{align}
where $i<j$ and $k<l$.
The algebra $\mathrm{M}_{p, u}(n)$ has a bialgebra structure under the comultiplication
$\mathrm{M}_{p, u}(n)\longrightarrow \mathrm{M}_{p, u}(n)\ot \mathrm{M}_{p, u}(n)$ given by
\begin{equation}
\Delta(t_{ij})=\sum_{k=1}^{n} t_{ik}\otimes t_{kj},
\end{equation}
and the counit given by $\varepsilon(t_{ij})=\delta_{ij}$, the Kronecker symbol. The coproduct can be
written simply as $\Delta(T)=T\dot{\ot}T$. This version of
quantum semigroup generalizes the two-parameter quantum case of \cite{T}.

For any permutation $\sigma$ in $S_n$, the $q$-inversion associated to the parameters $q_{ij}$ is defined as
\begin{equation}\label{e:qinv}
(-q)_{\sigma}=(-1)^{l(\sigma)}\prod_{\substack{i<j\\\si_i>\si_j}}q_{\sigma_{j}\sigma_{i}},
\end{equation}
where $l(\sigma)=|\{(i, j)|i<j, \si_i>\si_j\}|$ is the classical inversion number of $\sigma$. The $p$-inversion
number is defined similarly. In this paper, we use
the $v$-based quantum number defined by
\begin{equation}
[n]_v=1+v+\cdots+v^{n-1},
\end{equation}
and the quantum factorial $[n]_v!=[1]_v[2]_v\cdots [n]_v$ for any natural number $n\in\mathbb N$. In particular, $[0]!=1$.

We define the quantum row-determinant and column-determinant of $T$ as follows.
\begin{align}\label{e:qdet}
\rdet(T)&=\sum_{\sigma\in S_n}(-q)_{\sigma}t_{1,\sigma_1}\cdots t_{n,\sigma_n},\\ \label{e:qper}
\cdet(T)&=\sum_{\sigma\in S_n}(-p)_{\sigma}t_{\sigma_1,1}\cdots t_{\sigma_n,n}.
\end{align}
 The first property we show is that both are group-like elements:
\begin{align*}
\Delta(\rdet(T))&=\rdet(T)\otimes \rdet(T),\\
\Delta(\cdet(T))&=\cdet(T)\otimes \cdet(T).
\end{align*}

To see this we use two multi-parameter quantum exterior algebras associated to the parameters $p_{ij}$ and $q_{ij}$.
Set $\widehat{R}=PR$, and for a polynomial
$f(t)\in \mathbb{C}[t]$ we denote by $\mathcal{I}_{f,l}$ the two-sided ideal generated by the relation
\begin{align*}
f(\widehat{R})(X_1\ot X_2)=0,
\end{align*}
where $X=(x_1,\ldots,x_n)^t$. The quotient algebra
$$\mathbb{C}_{f,l}(X)=\mathbb{C}\langle x_1,\ldots,x_n\rangle/\mathcal{I}_{f,l}.
$$
becomes a left $\mathrm{M}_{p, u}(n)$-comodule \cite{FRT} under the structural map
$\mu_l:\mathbb{C}_{f,l}(X) \rightarrow \mathrm{M}_{p, u}(n)\ot \mathbb{C}_{f,l}(X)$ by $\mu_l(X )=T \dot{\ot} X$, where
$\dot{\ot}$ means the composition of tensor product with matrix product, explicitly
\begin{align}
\mu_l(x_{i})=\sum_{j=1}^n t_{ij}\ot x_{j}.
\end{align}

Similarly we denote by $\mathcal{I}_{f,r}$ the two-sided ideal generated by the relation
\begin{align*}
(Y_1^t\ot Y_2^t)f(\widehat{R})=0,
\end{align*}
where $Y=(y_1,\ldots,y_n)^t$.

The quotient algebra $\mathbb{C}_{f,r}(Y)=\mathbb{C}\langle y_1,\ldots,y_n\rangle/\mathcal{I}_{f,r}$
has a right comodule structure under the map $\mu_r:\mathbb{C}_{f,l}(Y) \rightarrow \mathbb{C}_{f,l}(Y) \ot \mathrm{M}_{p, u}(n)$
given by $\mu_r(Y^t)=Y^t \dot{\ot} T$.

\begin{theorem}
In the bialgebra $\mathrm{M}_{p, u}(n)$, one has that
$$\rdet (T)=\cdet (T).$$
\end{theorem}
\begin{proof}
Let $f=\widehat{R}+u^{-1}$. It is easy to see that the relations of
$\Lambda_q(X):=\mathbb{C}_{f,l}(X)$ are given by
\begin{align} \label{wedge1}
&x_j  x_i=-q_{ij}x_i x_j, \\
&x_i ^2=0,
\end{align}
where $1\leqslant i<j\leqslant n$. The relations of
$\Lambda_p(Y)=\mathbb{C}_{f,r}(Y)$ are
\begin{align}\label{wedge2}
&y_j  y_i=-p_{ij}y_i y_j, \\
&y_i ^2=0,\label{wedge4}
\end{align}
where $1\leqslant i<j\leqslant n$.

Applying $(\Delta\ot id)\mu_{l}=(id\ot \mu_{l} )\mu_{l}$ to $x_{i_1} \cdots  x_{i_t}$, one has that
\begin{equation}
{\rdet}_q(T^{i_1\ldots i_t}_{j_1\ldots j_t})=\sum_{k_1<\cdots <k_t}
{\rdet}_q(T^{i_1\ldots i_t}_{k_1\ldots k_t})\ot {\rdet}_q(T^{k_1\ldots k_t}_{j_1\ldots j_t}).
 \end{equation}
where $T^{i_1\ldots i_t}_{j_1\ldots j_t}$ refers to the $r\times r$-submatrix formed by the rows and columns of the top and lower indices.
Similarly \begin{equation}
{\cdet}_q(T^{i_1\ldots i_t}_{j_1\ldots j_t})=\sum_{k_1<\cdots <k_t}
{\cdet}_q(T^{i_1\ldots i_t}_{k_1\ldots k_t})\ot {\cdet}_q(T^{k_1\ldots k_t}_{j_1\ldots j_t}).
 \end{equation}
In particular, if $t=n$, one has that
\begin{align*}
\Delta(\rdet(T))&=\rdet(T)\otimes \rdet(T),\\
\Delta(\cdet(T))&=\cdet(T)\otimes \cdet(T).
\end{align*}

Consider the following linear element $\Phi$ in $\mathbb{C}_{f,r}(Y)\ot \mathrm{M}_{p, u}(n)\otimes \mathbb{C}_{f,l}(X)$:
\begin{equation}
\Phi=\sum_{i,j=1}^n y_i\ot t_{ij}\ot   x_j=Y^t\dot{\ot} T\dot{\ot} X.
\end{equation}

Put further $\delta=(\delta_1, \ldots, \delta_n)^T$,
$\partial=(\partial_1, \ldots, \partial_n)^T$, and consider the following elements:
\begin{align}
\delta &=T \dot{\ot} X, \\
\partial &=Y \dot{\ot} T.
\end{align}

Let $\omega_i=y_i\ot \delta_i=\sum_{j=1}^n y_i\ot t_{ij}\ot x_j$. The relations
\eqref{wedge1}-\eqref{wedge4} imply that
\begin{align}\label{e:ext1}
&\omega_i   \omega_i=0, \quad 1\leqslant i\leqslant n, \\ \label{e:ext2}
&\omega_j   \omega_i=u^2 \omega_i   w_j, \quad 1\leqslant i<j\leqslant n.
\end{align}

It follows from (\ref{e:ext1})-(\ref{e:ext2}) that
\begin{align*}
\Phi^n &=(\sum_{\sigma\in S_n}u^{2l(\sigma)})\omega_1 \cdots \omega_n \\
&=[n]_{u^2}!(y_1 \cdots  y_n)\ot (\partial_1 \cdots \partial_n)\\
&=[n]_{u^2}!(y_1 \cdots  y_n) \ot  \rdet(T)\ot  (x_1 \cdots  x_n).
\end{align*}

On the other hand, $\Phi$ can be rewritten as $\Phi=\sum_{i=1}^n\omega_i'$ with
$\omega_i'=\partial_i\ot x_i=\sum_{j=1}^n y_j\ot a_{ji}\ot x_i$.
Then
\begin{align*}
\Phi^n=[n]_{u^2}! (  y_1 \cdots y_n)\ot \cdet(T) \ot (x_1 \cdots  x_n),
\end{align*}
which implies that
$$\rdet(T)=\cdet(T).$$
\end{proof}

Due to this identity, from now on, we will define the multiparameter quantum determinant for the
quantum group as
\begin{equation}\label{e:qdet2}
\begin{aligned}
{\det}_q(T)&=\sum_{\sigma\in S_n}(-q)_{\sigma}t_{1,\sigma_1}\cdots t_{n,\sigma_n}\\
&=\sum_{\sigma\in S_n}(-p)_{\sigma}t_{\sigma_1,1}\cdots t_{\sigma_n,n}.
\end{aligned}
\end{equation}

For a pair of $t$ indices $i_1, \ldots, i_t$ and $j_1, \ldots, j_t$,
we define the quantum row-minor $\xi_{j_1 \ldots j_t}^{i_1 \ldots i_t}=\det_q(T_{j_1 \ldots j_t}^{i_1 \ldots i_t})$
as in (\ref{e:qdet2}).
Like the determinant, the quantum row minor row also equals to the quantum column minor
for any pairs of ordered indices $1\leqslant i_1<\cdots<i_t\leqslant n$ and $1\leqslant j_1<\cdots<j_t
\leqslant n$, which justifies the notation.

For any $t$ indices $i_1, \ldots, i_t$ 
\begin{align}\label{e:minor1}
\delta_{i_1} \cdots \delta_{i_t}
=\sum_{j_1<\cdots<j_t}{\det}_q(T^{i_1\ldots i_t}_{j_1\ldots j_t})
\ot x_{j_1} \cdots  x_{j_t},
\end{align}
where the sum runs through all indices $1\leqslant j_1<\cdots<j_t\leqslant n$. This
implies that ${\det}_q(T^{i_1\ldots i_t}_{j_1\ldots j_t})=0$ whenever there are two identical rows.

As $\delta_i$'s obey the relations \eqref{wedge1}-\eqref{wedge4}, for any $t$-shuffle $\sigma\in S_n$ : $1\leqslant \sigma_1<\cdots<\sigma_t, \sigma_{t+1}<\cdots<\sigma_n\leqslant n$,
one has that
$$\delta_{\sigma_1} \cdots \delta_{\sigma_t}  \delta_{\sigma_{t+1}} \cdots \delta_{\sigma_n}
=(-q)_{\sigma}\delta_1 \cdots \delta_n.
$$
Note that $x_j$'s also satisfy the same relations. This
then implies the following Laplace expansion by invoking \eqref{e:minor1}: for each fixed $t$-shuffle
$\sigma_1<\cdots<\sigma_t, \sigma_{t+1}<\cdots<\sigma_n$, one has that
\begin{equation}\label{e:lap1}
{\det}_q(T)=\sum_{\alpha}
\frac{(-q)_{\alpha}} {(-q)_{\sigma}}
{\det}_q(T_{\alpha_1 \ldots \alpha_t}^{\sigma_1 \ldots \sigma_t})
{\det}_q(T_{\alpha_{t+1} \ldots \alpha_n}^{\sigma_{t+1}\ldots \sigma_n}),
\end{equation}
where the sum runs through all $t$-shuffles $\alpha\in S_n$ such that $\alpha_1<\cdots<\alpha_t, \alpha_{t+1}<\cdots<\alpha_n$. The Laplace expansion generalizes the one-parameter
formula given by \cite{HH}.

In particular, for fixed $i, k$
\begin{equation}\label{e:Lap01}
\begin{aligned}
\delta_{ik}{\det}_q(T)&=\sum_{j=1}^n
\frac{\prod_{l<j}(-q_{lj})}{\prod_{l<i}(-q_{li})}t_{ij}
{\det}_q (T^{\hat{k}}_{\hat{j}})\\
&=\sum_{j=1}^n
\frac{\prod_{l>j}(-q_{jl})}{\prod_{l>i}(-q_{il})}
{\det}_q(T^{\hat{k}}_{\hat{j}})t_{ij},
\end{aligned}
\end{equation}
where $\hat{i}$ means the indices $1, \ldots, i-1, i+1, \ldots, n$ for brevity.

As for the quantum (column) determinant or column-minor, the corresponding Laplace expansion
for a fixed $r$-shuffle  $(\tau_1\ldots \tau_n)$ of $n$ such that $\tau_1<\cdots<\tau_r, \tau_{r+1}<\cdots<\tau_n$
\begin{equation}\label{e:lap2}
{\det}_q(T)
=\sum_{\beta}
\frac{(-p)_{\beta}} {(-p)_{\tau}}
{\det}_q(T^{\beta_1 \ldots \beta_r}_{\tau_1 \ldots \tau_r})
{\det}_q(T^{\beta_{r+1} \ldots \beta_n}_{\tau_{r+1}\ldots \tau_n}),
\end{equation}
where the sum runs through all $r$-shuffles $\beta\in S_n$ such that $\beta_1<\cdots<\beta_r, \beta_{r+1}<\cdots<\beta_n$.

In particular, we have that for fixed $i, k$
\begin{align}\nonumber
\delta_{ik}{\det}_q(T)&=\sum_{j=1}^n
\frac{\prod_{l<j}(-p_{lj})}{\prod_{l<i}(-p_{li})}
t_{ji}{\det}_q(T_{\hat{k}}^{\hat{j}})\\ \label{e:Lap2}
&=\sum_{j=1}^n\frac{\prod_{l>j}(-p_{jl})}{\prod_{l>i}(-p_{il})}
{\det}_q(T^{\hat{j}}_{\hat{k}})t_{ji}.
\end{align}

\begin{theorem}\label{quasi-central}
In the bialgebra $\mathrm{M}_{p, u}(n) $ one has that
$$t_{ij}{\det}_q(T)=u^{2(j-i)}\frac{\prod_{l=1}^{n}q_{li}}{\prod_{l=1}^{n}q_{lj}}{\det}_q(T) t_{ij}.$$
\end{theorem}
\begin{proof}
Let $X=(x_{ij}), T'=(t_{ij}'), T''=(t_{ij}'')$ be the matrices with entries in $\mathrm{M}_{p, u}(n) $ defined by
\begin{align}
x_{ij}&={\det}_q (T^{\hat{j}}_{\hat{i}}),\\
T'_{ij}&=\frac{\prod_{l<j}(-q_{lj})}{\prod_{l<i}(-q_{li})}t_{ij},\\
T''_{ij}&=\frac{\prod_{l>i}(-p_{il})}{\prod_{l>j}(-p_{jl})}t_{ij}.
\end{align}
It follows from the Laplace expansion that
$$T' {\det}_q=T'XT''={\det}_q T''.$$
Therefore $\frac{\prod_{l<j}(-q_{lj})}{\prod_{l<i}(-q_{li})}t_{ij}{\det}_q(T)
=\frac{\prod_{l>i}(-p_{il})}{\prod_{l>j}(-p_{jl})}{\det}_q(T) t_{ij}.$  This is exactly
\begin{equation}\label{det commute}
t_{ij}{\det}_q(T)=u^{2(j-i)}\frac{\prod_{l=1}^{n}q_{li}}{\prod_{l=1}^{n}q_{lj}}{\det}_q(T) t_{ij}.
\end{equation}
\end{proof}

\begin{remark}\cite{BG}
 ${\det}_q(T)$ is central if and only if $u^{2(j-i)}\prod_{l=1}^{n}q_{li}=\prod_{l=1}^{n}q_{lj}$ for any $i,j$.
\end{remark}

Theorem \ref{quasi-central} implies that ${\det}_q(T)$ is a regular element in the bialgebra $\mathrm{M}_{p, u}(n) $, therefore we can define the localization $\mathrm{M}_{p, u}(n) [{{\det}_q}^{-1}]$, which will be denoted as $\mbox{GL}_{p, u}(n)$ and called the multiparameter quantum
group corresponding to $\mathrm{GL}_n(\mathbb C)$. In fact, Theorem \ref{quasi-central} gives the following identity:
\begin{equation}\label{inverse of det}
{\det}_q(T)^{-1}t_{ij}=u^{2(j-i)}\frac{\prod_{l=1}^{n}q_{li}}{\prod_{l=1}^{n}q_{lj}} t_{ij}{\det}_q(T)^{-1}.
\end{equation}

By defining the antipode
\begin{equation}
\begin{aligned}
S(t_{ij})&=(-1)^{i-j}\frac{\prod_{l<i}q_{li}}{\prod_{l<j}q_{lj}}{\det}_q (T^{\hat{j}}_{\hat{i}}){\det}_q(T)^{-1}\\
&=(-1)^{i-j}\frac{\prod_{l>j}p_{jl}}{\prod_{l>i}p_{il}}{\det}_q(T)^{-1}{\det}_q(T^{\hat{j}}_{\hat{i}}),
\end{aligned}
\end{equation}
the bialgebra $\GL_{p, u}(n)$ 
becomes a Hopf algebra, thus a quantum group
in the sense of Drinfeld.

In fact, the second equation follows from (\ref{inverse of det}). Therefore, $TS(T)=S(T)T=I$ by the Laplace expansions. Subsequently
$$(\id\otimes S)\Delta=(S\otimes \id)\Delta=\varepsilon. $$

\section{Quasideterminants}

In this section we will work with the ring of fractions of noncommutative elements.
First of all let us recall some basic facts about quasideterminants.
  Let $X$ be the set of $n^2$ elements
$x_{ij}, 1\leq i,j\leq n$. For convenience, we also use $X$ to denote the matrix
$(x_{ij})$ over the ring generated by $x_{ij}$.

Denote by $F(X)$ the free division ring generated by $0,1,x_{ij}, 1\leq i,j\leq n$.
It is well-known that the matrix $X=(x_{ij})$ is an invertible element over $F(X)$
 \cite{GGRW} .


Let $I,J$ be two finite subsets of cardinality $k\leq n$ inside $\{1, \ldots, n\}$.
Following \cite{GGRW}, we introduce the notion of quasiderminant.
\begin{definition}
For $i\in I, j\in J$, the $(i,j)$-th quasideterminant $|X|_{ij}$ 
is the following element of $F(X)$:
$$|X|_{ij}=y_{ji}^{-1},$$
where $Y=X^{-1}=(y_{ij})$.
\end{definition}

If $n=1$, $I=i, J=j$. Then $|X|_{ij}=x_{ij}$.

When $n\geq 2$, and let $X^{ij}$ be the $(n-1)\times(n-1)$-matrix obtained from X by
deleting the $i$th row and $j$th column. In general $X^{i_1\cdots i_r, j_1\cdots j_r}$ denotes
the submatrix obtained from $X$ by deleting the $i_1, \cdots, i_r$-th rows, and
$i_1, \cdots, i_r$-th columns. Then
$$|X|_{ij}=x_{ij}-\sum_{i', j'} x_{ii'}(|X^{ij}|_{j'i'})x_{j'j},$$
where the sum runs over $i'\notin I\setminus\{i\}, j'\notin J\setminus\{j\}$.

\begin{theorem}
In the ring of fractions of elements of $\mathrm{M}_{p, u}(n)$, one has that
 \begin{equation}\label{det quasidet1}
 {\det}_q(T)=|T|_{11}|T^{11}|_{22}|T^{12,12}|_{33}\cdots t_{nn}
 \end{equation}
and the quasi-minors on the right-hand side commute with each other. More generally, for
 two permutations $\sigma$ 
and $\tau$ 
of $S_n$, one has that
 \begin{equation}\label{det quasidet2}
 {\det}_q(T)=\frac{(-q)_{\tau}}{(-q)_{\sigma}}|T|_{\si_1\tau_1}|T^{\si_1\tau_1}|_{\si_2\tau_2}|T^{\si_1\si_2,\tau_1\tau_2}|_{\si_3\tau_3}\cdots t_{\si_n\tau_n}.
 \end{equation}
\end{theorem}

\begin{proof}
By definition the quasi-determinants of $T$ are  inverses of the entries of the antipode $S(T)$,
\begin{equation}
\begin{aligned}
|T|_{ij}=S(t_{ji})^{-1}&=\frac{\prod_{l<i}(-q_{li})}{\prod_{l<j}(-q_{lj})}{\det}_q(T){\det}_q (T^{ij})^{-1}\\
&=\frac{\prod_{l>j}(-p_{jl})}{\prod_{l>i}(-p_{il})}{\det}_q(T^{ij})^{-1}{\det}_q(T),
\end{aligned}
\end{equation}
then
\begin{equation}
\begin{aligned}
\frac{\prod_{l<j}(-q_{lj})}{\prod_{l<i}(-q_{li})}|T|_{ij}{\det}_q (T^{ij})={\det}_q(T), \\
\frac{\prod_{l>i}(-p_{il})}{\prod_{l>j}(-p_{jl})} {\det}_q(T^{ij})|T|_{ij}= {\det}_q(T).
\end{aligned}
\end{equation}
By induction on the size of the matrix $T$, one sees that (\ref{det quasidet1}) and (\ref{det quasidet2}) hold.

It follows from \eqref{det commute} that ${\det}_q(T^{\{1,\ldots,s\}\{1,\ldots,s\}})$ and
 ${\det}_q(T^{\{1,\ldots,t\}\{1,\ldots,t\}})$ commute for $1\leq s,t\leq n-1$. Any factor on the right hand side of \eqref{det quasidet1} can be expressed as ${\det}_q (T^{\{1,\ldots,s\}\{1,\ldots,s\}}) {{\det}_q (T^{\{1,\ldots,s+1\}\{1,\ldots,s+1\}})}^{-1}$ multiplied by a scalar,
 therefore they commute with each other.
\end{proof}

\section{Multi-parameter quantum algebras}
In this section we study the algebra $U(R)=\mathrm{U}_{p, u}(\mathfrak{gl}_n)$ as the dual Hopf algebra of $\mbox{GL}_{p, u}(n)$. Following \cite{FRT, RTF},
$U(R)$ is defined as the associative algebra generated by
$l_{ij}^{+}$, $l_{ji}^{-}$, $1\leq i \leq j \leq n$ with unit subject to the relations:
\begin{align}\label{RTT relation 1}
R^+L_1^{\pm}L_2^{\pm}=L_2^{\pm}L_1^{\pm}R^+,\\ \label{RTT relation 2}
R^+L_1^{+}L_2^{-}=L_2^{-}L_1^{+}R^+,
\end{align}
where $L^{\pm}=(l_{ij}^{\pm})_{n\times n}$, $L_1^{\pm}=L^{\pm}\otimes I$ and $L_2^{\pm}=I\ot L^{\pm}$.
$U(R)$ is a Hopf algebra with the coproduct, counit and antipode given by
\begin{align*}
&\Delta(L^{\pm})=L^{\pm}\dot{\ot} L^{\pm},\\
&\varepsilon (L^{\pm})=I,\\
&S(L^{\pm})=(L^{\pm})^{-1}.
\end{align*}
There is a unique pairing of Hopf algebras
\begin{equation}
( \ \ ,\ \ ):U(R)\times \GL_{p,u}(n)\rightarrow \mathbb C
\end{equation}
satisfying the following relations
\begin{align}\label{hopf pair}
(L_{1}^{\pm}, T_2)=R^{\pm}.
\end{align}

In terms of generators, the pairing is given by
\begin{align}
&(l_{kk}^{\pm},t_{ij})=\delta_{ij} u^{\pm} , k=j ,\\
&(l_{kk}^{\pm}, t_{ij})=\delta_{ij} q_{jk}u^{-1} ,  k>j,\\
&(l_{kk}^{\pm},t_{ij})=\delta_{ij} p_{kj} u^{-1},  k<j,\\
&(l_{k,l}^{+},t_{ij})=\delta_{k,j}\delta_{i,l} ( u-u^{-1}),k<l,  \\
&(l_{l,k}^{-},t_{ij})=\delta_{l,j} \delta_{i,k}(u^{-1}-u ) ,k<l.
\end{align}
The pairing is extended to the whole algebra as follows: for any $a, b\in U(R), u, v\in \GL_{p,u}(n)$
\begin{align}
(ab, u)&=(a\ot b, \Delta(u)),\\
(a, uv)&=(\Delta(a), u\ot v).
\end{align}
It is easy to see that
\begin{align}
&(l_{k,l}^{\pm},{\det}_q(T))=0 , k\neq l , \\
&(l_{k,k}^{+},{\det}_q(T))=u^{2-n}\prod_{j<k} q_{jk}  \prod_{j>k} p_{kj} ,\\
&(l_{k,k}^{-},{\det}_q(T))=u^{n-2} \prod_{j<k} p_{jk}^{-1} \prod_{j>k} q_{kj}^{-1}.
\end{align}

The elements of $U(R)$ can be viewed as linear functionals on
$\GL_{p,u}(n)$ in a canonical method. If $V$ is a right $\GL_{p,u}(n)$-comodule (resp. left
$\GL_{p,u}(n)$-comodule) with the structure map $L_G:V\rightarrow
V\otimes \GL_{p,u}(n)$ (resp. $L_G:V\rightarrow  \GL_{p,u}(n)\otimes
V$), then $V$ has a left (resp. right) $U(R)$-module structure via
\begin{equation}
a.v=(id\otimes a)R_G(v) \ \ (resp.\ v.a=(a\otimes id)L_G(v)),
\end{equation}
for all $a\in U(R)$ and $v\in V$. Thus
\begin{align}\label{e:action1}
L^{\pm}_1.T_2=R^{\pm}T_2,\\ \label{e:action2}
T_1.L_2^{\pm}=T_1R^{\pm}.
\end{align}
So the action on the generators is given by
\begin{align}
&l_{kk}^{\pm}.t_{ik}=u^{\pm 1} t_{ik}, \\ 
&l_{kk}^{\pm}.t_{ij}=u^{-1}q_{jk}t_{ij},  k>j,\\ 
&l_{kk}^{\pm}.t_{ij}=u^{-1} p_{kj} t_{ij},  k<j,\\
&l_{kl}^{+}.t_{ij}=\delta_{k,j} (u-u^{-1})t_{il},k<l,  \\
&l_{lk}^{-}.t_{ij}=\delta_{l,j} (u^{-1}-u )t_{ik},k<l.
\end{align}

Similarly the right action of $U(R)$ is given by
\begin{align}
&t_{kj}.l_{kk}^{\pm}=u^{\pm 1} t_{kj},\\
&t_{ij}.l_{kk}^{\pm}=u^{-1}q_{ik} t_{ij},  k>i,\\
&t_{ij}.l_{kk}^{\pm}=u^{-1}p_{ki}t_{ij},  k<i,\\
&t_{ij}.l_{k,l}^{+}=\delta_{l,i} (u-u^{-1})t_{k,j}, k<l, \\
&t_{ij}.l_{l,k}^{-}=\delta_{k ,i} (u^{-1}-u )t_{l,j},k<l.
\end{align}

Moreover the action on ${\det}_q(T)$  is as follows:
\begin{align}
&l_{kl}^{\pm}.{\det}_q(T)={\det}_q(T). l_{k,l}^{\pm}=0, k\neq l ,\\
&l_{kk}^{+}.{\det}_q(T)={\det}_q(T).l_{k,k}^{+}=u^{2-n} \prod_{j<k} q_{jk} \prod_{j>k} p_{kj}  {\det}_q(T)  ,\\
&l_{kk}^{-}.{\det}_q(T)={\det}_q(T).l_{k,k}^{-}=u^{n-2} \prod_{j<k} p_{jk}^{-1} \prod_{j>k} q_{kj}^{-1}  {\det}_q(T).
\end{align}


For the spectral parameter $\lambda$, we introduce the $R$-matrix $R(\lambda)=\lambda R^+ -{\lambda} ^{-1}R^{-}$,
which satisfies the Yang-Baxter equation:
\begin{equation}\label{YBE}
R_{12}(\lambda/\mu)R_{13}(\lambda)R_{23}(\mu)=R_{23}(\mu)R_{13}(\lambda)R_{12}(\lambda/\mu).
\end{equation}
Set $L(\lambda)=\lambda L^+ -\lambda^{-1}L^{-}$. Then the commutation relations \eqref{RTT relation 1} and
\eqref{RTT relation 2} can be combined into one: 
\begin{equation}\label{YBE2}
R(\lambda/\mu)L_1 (\lambda)L_2 (\mu)=L_2 (\mu)L_1 (\lambda)R(\lambda/\mu).
\end{equation}

As $U(R)$ acts on the quantum group from both sides, we will adopt the following convention to distinguish the actions. For any element $a$ of $U(R)$ or $\mathrm{U}_{p, u}(\mathfrak{gl}_n)$,  
we denote by $a$ and $a^{\mathrm{o}}$ its action from the left and right respectively when viewed as operators.

In particular, we consider the left (right) action of the elements $l_{ij}^{\pm}$ on $\mathrm{GL}_{p, u}(n)$. Mimicking the
classical action of $\mathrm{GL}_n(\mathbb C)\times\mathrm{GL}_n(\mathbb C)$ on the algebra of polynomial functions
$\mathcal P(\mathrm{M}_n(\mathbb C))$, we introduce
the following special linear operator $L(\lambda)S(T)^{\circ}\in \mathrm{End}(\mathrm{GL}_{p, u}(n))$. Explicitly its action on any $\phi\in \mathrm{GL}_{p, u}(n)$ is given by
\begin{equation}
L(\lambda)S(T)^{\circ}(\phi)=L(\lambda).(\phi.S(T))
\end{equation}
The components of $L(\lambda)S(T)^{\circ}$ behave similarly as the classical 
differential operators $\frac{\partial}{\partial t_{ij}}$.
Similarly, one can also consider the operator $S(T)L(\lambda)^{\circ}$. It turns out that these two operators are the same.

\begin{proposition}\label{P:diff} For any spectral parameter $\lambda$, the following identity holds over $\GL_{p, u}(n)$
\begin{align}\label{e:par1}
L(\lambda)S(T)^{\circ}=S(T)L(\lambda)^{\circ}.
\end{align}
Set $L(\lambda)S(T)^{\circ}=(u-u^{-1})\partial(\lambda)^t$, and $\partial(\lambda)=\lambda\partial^{+}-\lambda^{-1}\partial^-$, then
the matrix entries $\partial_{ij}^{\pm}$ of $\partial^{\pm}$ are the linear operators in $\mathrm{End}(\GL_{p,u}(n))\otimes \mathbb{C}[u,\ u^{-1}]$
that satisfy the following relations:
\begin{align}\label{eq1}
 L_{ji}( \lambda)=( u -u^{-1} )\sum_{k=1}^{n}t_{ki}^{\mathrm{o}}\partial_{kj}(\lambda),\\ \label{eq2}
 L_{ji}^{\mathrm{o}}(\lambda )=( u -u^{-1} )\sum_{k=1}^{n}t_{jk}\partial_{ik}(\lambda).
\end{align}
\end{proposition}
\begin{proof}
The identity \eqref{e:par1} means that for any $\varphi\in \GL_{p, u}(n)$,
\begin{equation}\label{eq par1}
\begin{split}
( u-u^{-1} ) \partial(\lambda)^t \varphi  &=( L(\lambda).\varphi ) S(T)\\
&=S(T)(\varphi.L(\lambda) ).
\end{split}
\end{equation}
Multiplying $T$ to both sides of \eqref{eq par1}, the identity becomes the following one:
\begin{align}\label{eq par}
TL(\lambda).\varphi=\varphi.L(\lambda)T,
\end{align}
or equivalently, $TL^{\pm}.\varphi=\varphi.L^{\pm}T$ on $\GL_{p,u}(n)$. This latter identity can be shown as follows.
For the generators $t_{ij}$, by the action of $L^{\pm}$ \eqref{e:action1}-\eqref{e:action2} we have that
\begin{align}
T_{1}L_1.T_{2}=T_{1}T_{2}R^{\pm}=R^{\pm}T_{2}T_{1}=T_{2}.L^{\pm} T_{1}.
\end{align}

The validity of \eqref{eq par} in general follows from the coproduct. In fact for any two elements $\varphi, \psi\in \GL_{p,u}(n)$, it follows from the action of $L^{\pm}$ that
\begin{equation}
\begin{split}
TL^{\pm}.(\varphi\psi)&=T\Delta(L^{\pm}).(\varphi\ot \psi)\\
&=T(L^{\pm}.\varphi)(L^{\pm}.\psi)\\
&=(\varphi.L^{\pm})T(L^{\pm}.\psi)\\
&=(\varphi.L^{\pm})(\psi.L^{\pm})T\\
&=(\varphi\psi).L^{\pm}T.
\end{split}
\end{equation}
Therefore \eqref{e:par1} is proved and the operators $\partial(\lambda)$ are uniquely determined by the expression \eqref{eq par1}.
\end{proof}

Set $\widehat{R}(\lambda)=PR(\lambda)$.
Then the Yang-baxter equation \eqref{YBE} is equivalent to the braid
relation
\begin{equation}\label{YBE1}
\widehat{R}_{12}(\lambda/\mu)\widehat{R}_{23}(\lambda)\widehat{R}_{12}(\mu)
=\widehat{R}_{23}(\mu)\widehat{R}_{12}(\lambda)\widehat{R}_{23}(\lambda/\mu).
\end{equation}
Correspondingly the defining relation \eqref{YBE2} can be written as
\begin{equation}\label{YBE2'}
\widehat{R}(\lambda/\mu)L(\lambda)\ot L(\mu)
=L(\mu)\ot L(\lambda)\widehat{R}(\lambda/\mu).
\end{equation}

For any $i<j$, $\widehat{R}(u^{-1}) e_j\ot e_i= -p_{ij}^{-1} \widehat{R}(u^{-1}) e_i\ot e_j$.
We introduce the antisymmetrizers inductively as follows. Let $s_2=\widehat{R}_{12}(u^{-1})$,
and iteratively define that
$$
s_{k+1}=\widehat{R}_{12}(u^{-1})\widehat{R}_{23}(u^{-2})\cdots\widehat{R}_{k,k+1}(u^{-k})
s_k.$$

We denote by $A^{(n)}$ the normalized antisymmetrizer:
\begin{equation}
A^{(n)}=\frac{1}{[n]_{u^2}!}\sum_{\sigma,\tau \in S_n}(-q)_{\sigma}^{-1}(-p)_{\tau}^{-1}
e_{\sigma(1)\tau(1)}\ot \cdots \ot e_{\sigma(n)\tau(n)},
\end{equation}
then $(A^{(n)})^2=A^{(n)}$.
The following lemma can be verified by induction.

\begin{lemma}\label{per1} One has that
\begin{equation}
s_n=(u-u^{-1})^{\frac{n(n-1)}{2}}[n]_{u^2}! A^{(n)}
\end{equation}
\end{lemma}

\begin{lemma}\label{per2} One has the identity
\begin{equation}
\begin{aligned}
&A^{(n)}R_{0n}(x u^{-(n-1)})\cdots R_{02}(x u^{-1}) R_{01}(x)\\
&=R_{0n}(x)\cdots R_{02}(x u^{-(n-2)})) R_{01}(x u^{-(n-1)})
A^{(n)}\\
&=M(x)_0A^{(n)}
\end{aligned}
\end{equation}
where $M=(M_{ij})\in \mathrm{End}(\mathbb C^{n})$ is a diagonal matrix such that
\begin{align}
M_{ii}(x)=u^{-n+1}(xu-x^{-1}u^{-1})\prod_{j\neq i}(xu^{2-j}-x^{-1}u^{j-2})\prod_{j<i}q_{ji}\prod_{j>i}p_{ij}.
\end{align}
\end{lemma}

\begin{proof}
The first equality follows from the Yang-Baxter equation.
Since for any $\sigma\in S_n$
$$A^{(n)} e_{\sigma(1)}\otimes\cdots\otimes e_{\sigma(n)}=(-p)_{\sigma}^{-1}A^{(n)}
e_1\otimes e_2\otimes \cdots\otimes e_n,$$
it is enough to prove that
\begin{align*}
A^{(n)}R_{0n}(x u^{-(n-1)})\cdots R_{02}(x u^{-1}) R_{01}(x)e_i\ot e_i\ot e_1\otimes \cdots \hat{e_i}\otimes \cdots\otimes e_n\\
=M_{ii}e_i\ot A^{(n)} e_i\ot e_i\ot e_1\otimes \cdots \hat{e_i}\otimes \cdots\otimes e_n.
\end{align*}
Note that $R_{01}(x)e_i\ot e_i=(xu-x^{-1}u^{-1}) e_i\ot e_i$ and
$R_{0n}(x u^{-(n-1)})\cdots R_{02}(x u^{-1})$ fix the first component. The only vector not annihilated by $A^{(n)}$ is $e_i\ot e_i\ot e_1\otimes \cdots \hat{e_i}\otimes \cdots\otimes e_n$, and the coefficient is $M_{ii}(x)$.
\end{proof}

\begin{theorem}
Let $z(x)$ be the following element of $\mathrm{U}_{p, u}(\mathfrak{gl}_n)$:
\begin{align}\label{center}
z(x)&=\sum_{\sigma,\tau\in S_n}(-p)_{\sigma}^{-1}(-q)_{\tau}^{-1}
L_{\sigma_{1}\tau_{1}} (x u^{1-n})
L_{\sigma_{2}\tau_{2}} (x u^{2-n})\cdots
L_{\sigma_{n}\tau_{n}}(x )
\end{align}
Then the coefficients $z_i$ of $z(x)$ are quasi-central elements of
$\mathrm{U}_{p, u}(\mathfrak{gl}_n)$ in the sense that
\begin{equation}
L(\lambda)z(x)=z(x)M(\lambda u^{n-1}/x) ^{-1} L (\lambda)M( \lambda u^{n-1}/x).
\end{equation}
\end{theorem}
\begin{proof} It follows from \eqref{YBE2}, Lemma \ref{per1} and Lemma \ref{per2} that
\begin{equation*}
\begin{split}
&L_0 (\lambda) A^{(n)}   L_1 (x u^{1-n}) \cdots L_{n} (x  )
\\
=&  R_{01}(\lambda /x)^{-1} R_{02}(\lambda u/x)^{-1}\cdots R_{0n}(  \lambda u^{n-1}/x)^{-1}\\
&\times L_1 (x)L_2 (x u^{-1})\cdots L_{n} (x u^{1-n})L_0 (u)\\
&\times R_{0n}( \lambda u^{n-1}/x)\cdots R_{01}(\lambda /x)A^{(n)}
\\
=&  R_{01}(\lambda /x)^{-1} R_{02}(\lambda u/x)^{-1}\cdots R_{0n}(  \lambda u^{n-1}/x)^{-1}\\
&\times L_1 (x)L_2 (x u^{-1})\cdots L_{n} (x u^{1-n})L_0 (u)M(\lambda u^{n-1}/x)_0A^{(n)}
\\
=&  R_{01}(u /x)^{-1} R_{02}(u \lambda /x)^{-1}\cdots R_{0n}(u \lambda^{n-1}/x)^{-1}A^{(n)} \\
&\times L_1 (x u^{1-n}) \cdots L_{n} (x  )L_0 (u)M(\lambda u^{n-1}/x)_0
\\
=&A^{(n)}   L_1 (x u^{1-n}) \cdots L_{n} (x  ) M(\lambda u^{n-1}/x)_0^{-1}  L_0 (u)M(\lambda u^{n-1}/x)_0
\end{split}
\end{equation*}
Taking partial trace $\mathrm{Tr}_{12\ldots n}$ on both sides, we obtain that
\begin{equation}
L(\lambda)z(x)=z(x)M(\lambda u^{n-1}/x) ^{-1} L (\lambda)M( \lambda u^{n-1}/x).
\end{equation}
\end{proof}

Recall that the quantum $p$-exterior algebra $\Lambda_n(Y)$ is generated by $y_{1}, \ldots, y_{n}$ subject to the relations
$Y_1^t Y_2^t(\widehat{R}+u^{-1})=0$ or explicitly
\begin{align}
&y_{i}^{2}=0, \\
&p_{ij}y_{i}y_{j}+y_{j}y_{i}=0\ \ (1\leq i<j\leq n)\ .
\end{align}

We consider the following elements $\eta_{j}, \zeta_{j}^{\pm}\in\Lambda_{n}(Y)\otimes \mathrm{End}(\mathrm{M}_{p, u}(n))$ defined by
\begin{align}\label{e:vectH}
& \eta_{j}=\sum_{i=1}^{n}y_{i}\otimes t_{ij},\\ \label{e:vectZ}
&
\zeta_{j}^{\pm}=(u-u^{-1})^{-1}\sum_{i=1}^{n}y_{i}\otimes L_{ij}^{\pm\circ}.
\end{align}
Set $H=(\eta_1, \ldots, \eta_n)^t$ and $Z^{\pm}=(\zeta^{\pm}_1, \ldots, \zeta^{\pm}_n)^t$ . Then the equations \eqref{e:vectH}-\eqref{e:vectZ} can be written as
\begin{align}
H^t&=Y^t\dot{\ot}T, \\
{Z^{\pm}}^t&=(u-u^{-1})^{-1}Y^t\dot{\ot}{L^{\pm}}^{\circ}.
\end{align}
Therefore
$Z(\lambda)^t=(u-u^{-1})^{-1}Y^t\dot{\ot}{L(\lambda)}^{\circ}=(\zeta_1(\lambda), \ldots, \zeta_n(\lambda))$, or
by \eqref{eq1}
\begin{align}
\zeta_{j}(\lambda)=\sum_{i,k}y_i\otimes t_{ik}\partial_{jk}(\lambda)=\sum_{k}\eta_{\mathrm{k}}\partial_{jk}(\lambda)\ .
\end{align}

\begin{proposition}\label{comm relation}
For the elements $\eta_j$, $\zeta_{j}(\lambda)$, the following commutation relations hold:
\begin{align}\label{e:comm1}
& \zeta_{i}(u\lambda)\zeta_{i}(\lambda)=0\quad (i=1,\ldots,n),\\ \label{e:comm2}
&  p_{ij}\zeta_{i}(u\lambda) \zeta_{j}(\lambda)+\zeta_{j}(u\lambda)\zeta_{i}(\lambda)=0 \quad (1\leq i<j\leq n),\\ \label{e:comm3}
&  \zeta_{i}(u\lambda)\eta_{j}+\eta_{j}\zeta_{i}(\lambda)=0 \quad (1\leq i,j\leq n).
\end{align}
\end{proposition}
\begin{proof} Using the variables $\zeta_i(\lambda)$ to form the
vector $Z(\lambda)=(\zeta_1(\lambda), \ldots, \zeta_n(\lambda))^t$, we can
write the first two relations as
\begin{align*}
Z(u\lambda)^t\ot Z(\lambda)^t\widehat{R}(u)=0.
\end{align*}
Note that $H=(\eta_1, \ldots, \eta_n)^t$, the last relation is rewritten as
\begin{align}
Z(u\lambda)^t\ot H^t=-H^t\ot Z(\lambda)^t
\end{align}

We can show the relation easily as follows.
\begin{align*}
Z(u\lambda)^t\ot Z(\lambda)^t\widehat{R}(u)&=(u-u^{-1})^{-2}(Y^t\dot{\ot}L(u\lambda)^{\circ})\ot (Y^t\dot{\ot}L(\lambda)^{\circ})\widehat{R}(u)\\
&=(u-u^{-1})^{-2}(Y^t\ot Y^t)\dot{\ot} (L(u\lambda)^{\circ}\ot L(\lambda)^{\circ})\widehat{R}(u)\\
&=(u-u^{-1})^{-2}(Y^t\ot Y^t)\dot{\ot}\widehat{R}(u)(L(\lambda)^{\circ}\ot L(u\lambda)^{\circ})\\
&=(u-u^{-1})^{-2}(Y^t\ot Y^t)\widehat{R}(u)\dot{\ot}(L(\lambda)^{\circ}\ot L(u\lambda)^{\circ})=0
\end{align*}

Recalling the defining relation, we have that
\begin{align*}
(u-u^{-1})&H^t_{2}{Z^{\pm}_{1}}^t=y^t_{2}y^t_{1}T_{2}L_{1}^{\pm\circ}\\
=&y^t_{2}y^t_{1}(R^{\pm})^{-1}L_{1}^{\pm\circ}T_{2}
\\
 =&-u^{\pm 1}y^t_{1}y^t_{2}L_{1}^{\pm\circ}T_{2} \\
 =&-(u-u^{-1})u^{\pm 1}{Z_{1}^{\pm}}^tH^t_{2}.
\end{align*}
This completes the proof.
\end{proof}

Up to a scalar, $z(x)$ can be written as follows:
\begin{equation}
\sum_{\sigma,\tau\in S_n}(-p)_{\sigma} (-q)_{\tau}
L_{\sigma_{n}\tau_{n}} (x u^{1-n})
L_{\sigma_{n-1}\tau_{n-1}} (x u^{2-n})\cdots
L_{\sigma_{1}\tau_{1}}(x ).
\end{equation}
This latter element will be denoted by $c(x)$. Then
\begin{equation}
c(x)^{\circ}=\sum_{\sigma,\tau\in S_n}(-p)_{\sigma} (-q)_{\tau}
L_{\sigma_{1}\tau_{1}}(x )^{\circ}
L_{\sigma_{2}\tau_{2}} (x u^{-1})^{\circ}\cdots
L_{\sigma_{n}\tau_{n}} (x u^{1-n})^{\circ}
.
\end{equation}
By Proposition \ref{comm relation}, $c(x)^{\circ}$ can be simplified as follows:
\begin{equation}
c(x)^{\circ}=[n]_{u^2}!\sum_{\sigma\in S_n}(-p)_{\sigma}
L_{\sigma_{1}1 }(x )^{\circ}
L_{\sigma_{2}2 } (x u^{-1})^{\circ}\cdots
L_{\sigma_{n}n } (x u^{1-n})^{\circ}
.
\end{equation}
Let $d(x)=(u-u^{-1})^{-n}[n]_{u^2}!^{-1}c(x)^{\circ}$.
In terms of $\zeta_{j}$'s, $d(x)$ is expressed as an element in the quantum exterior algebra.
\begin{align}\label{e:cap1}
\zeta_{1}(x)\zeta_{2}(x u^{-1})\cdots\zeta_{n}(xu^{1-n})
=y_{1}\cdots y_{n}\otimes  d(x)\ .
\end{align}

\begin{theorem}[Generalized quantum Capelli identity] On the multiparameter quantum group $\GL_{p,u}(n)$
one has the following generalization of the classical Capelli identity:
$$d(x u^{n-1}) =  \prod_{i<j}p_{ij} {\det}_p(T) {\rdet}_{p^{-1}}(\partial (x ) ).$$
\end{theorem}
\begin{proof} Using the commutation relation \eqref{e:comm3} of $\eta_j$ and $\xi_i(\lambda)$, we have that
\begin{align*}
&\zeta_{1}(xu^{n-1})\zeta_{2}(x u^{n-2})\cdots\zeta_{n}(x)\\
&=\sum_{k}\zeta_{1}(xu^{n-1})\zeta_{2}(x u^{n-2})\cdots\zeta_{n-1}(xu) \eta_{k}\partial_{nk}(x )\\
&=(-1)^{n-1}\sum_{k}\eta_{k}\zeta_{1}(xu^{n-2}) \cdots\zeta_{n-1}(x ) \partial_{nk}(x )\\
&=(-1)^{ \frac{n(n-1)}{2}}\sum_{k_n,\ldots,k_n}
\eta_{k_n}\cdots\eta_{k_1}
\partial_{1k_1}(x ) \cdots \partial_{nk_n}(x )\\
&=(-1)^{ \frac{n(n-1)}{2}}
\eta_{n}\cdots\eta_{1}
{\rdet}_{p^{-1}}(\partial (x ) )\\
&=\prod_{i<j}p_{ij}
\eta_{1}\cdots\eta_{n}
{\rdet}_{p^{-1}}(\partial (x ) )\\
&=\prod_{i<j}p_{ij}y_1\cdots y_n \ot
{\det}_p(T)
{\rdet}_{p^{-1}}(\partial (x ) ).\\
\end{align*}

On the other hand, by \eqref{e:cap1} this is equal to $y_{1}\cdots y_{n}\otimes d(x u^{n-1})$, which completes the proof.
\end{proof}

For each $k=1, \ldots, n$ we set
\begin{align*}
\alpha_k&=u^{2-n}\prod_{j<k} q_{jk} \prod_{j>k} p_{kj},\\
\beta_k&=u^{n-2}\prod_{j<k} p_{jk}^{-1} \prod_{j>k} q_{kj}^{-1}.
\end{align*}
and for the spectral parameter $x$, denote $\gamma_{k,s}(x)=x \alpha_k^s -x^{-1}\beta_k^s$, $s\in\mathbb N$.

\begin{corollary} For each $s\in\mathbb N$ one has that
$$   {\rdet}_{p^{-1}}(\partial (x ) ) {{\det}_p(T)}^s=  (u-u^{-1})^{-n}\prod_{i<j}p_{ij}^{-1}\prod_{i=1}^n\gamma_{i,s}(xu^{n-i}){{\det}_q(T)}^{s-1}.$$

\end{corollary}
\begin{proof}
Note that
\begin{align*}
&{\det}_{q}(T).L_{ij}^{\pm}=0, i \neq j,\\
&{\det}_q .l_{k,k}^{+}=\alpha_k {\det}_q  ,\\
&{\det}_q. l_{k,k}^{-}=\beta_k {\det}_q.
\end{align*}
Therefore
$\det_{q}(T)^{s}.L_{ii}(x)= \gamma_{k,s}(x)\det_{q}(T)^{s}$. Applying both sides of the Capelli identity to $\det_{q}(T)^{s}$, we obtain the formula.
\end{proof}

\section{Multiparameter quantum Pfaffians}

\begin{definition}
Let $B=(b_{ij})$ be an $2n\times 2n$ square $p$-antisymmetric matrix with noncommutative entries such that
$b_{ji}=-p_{ij}b_{ij}, i<j$. The
multiparameter quantum $q$-Pfaffian is defined by
\begin{align*}
\Pf_q(B)
=\sum_{\sigma\in \Pi}(-q)_{\sigma}b_{\sigma(1)\sigma(2)}b_{\sigma(3)\sigma(4)}\cdots b_{\sigma(2n-1)\sigma(2n)},
\end{align*}
where $p=(p_{ij}),q=(q_{ij}),i<j$, and the sum runs through the set $\Pi$ of permutations $\sigma$ of $2n$ such that
$\sigma(2i-1)<\sigma(2i), i=1,\ldots,n.$
\end{definition}

Note that the parameters $q_{ij}$ and $p_{ij}$ satisfy the $(p, u)$ condition: $p_{ij}q_{ij}=u $.
\begin{proposition}\label{Pfaffian-Lap}
For any $0\leq t\leq n$,
\begin{equation}
\Pf_q(B)=\sum_{I} inv(I,I^c) \Pf_q(B_{I})\Pf_q(B_{I^c}),
\end{equation}
where the sum is taken over all subsets $I=\{i_1\cdots i_{2t}|i_1<\cdots <i_{2t}\}$ of $[1,2n]$, and
\begin{equation}
inv(I,J)=\prod_{i\in I, j\in J, i>j} (-q_{ji}).
\end{equation}
\end{proposition}
\begin{proof}
Let $\Omega=\sum_{i<j}b_{ij}x_{i}x_{j}$, where $x_i\in\Lambda_q(x)$. Then
\begin{equation}
\Omega^n=\Pf_q(B)x_{1} x_2\cdots  x_{2n}.
\end{equation}

On the other hand,
\begin{equation}
\begin{aligned}
\Omega^n&=\Omega^t \Omega^{n-t}\\
&=\sum_{I,J}\Pf(B_{I})x_I\Pf(B_{I^c})x_{J}\\
&=\sum_{I,J}\Pf(B_{I})\Pf(B_{I^c})x_Ix_{J}.
\end{aligned}
\end{equation}
It is easy to see that $x_Ix_{J}$ vanishes unless $J=I^c$. Therefore
\begin{align*}
& \Omega^n
=\sum_{I}\Pf(B_{I})\Pf(B_{I^c})x_Ix_{I^c}\\
&=\sum_{I}inv(I,I^c)\Pf(B_{I})\Pf(B_{I^c})x_{1}x_2\cdots x_{2n}.
\end{align*}
Thus we conclude that
$$\Pf(B)=\sum_{I}inv(I,I^c)\Pf(B_{I})\Pf(B_{I^c}).$$
\end{proof}

\begin{theorem}
Let $B=(b_{ij})_{1\leq i,j \leq 2n}$ be the $p$-antisymmetric matrix such that $b_{ji}=-p_{ij}b_{ij}, i<j$, and assume that
 the entries of $B$ commute with those of a $(p, u)$-matrix $T=(t_{ij})_{1\leq i,j\leq 2n}$. Let $C=T^{t}BT$. Then
\begin{equation}
c_{ji}=-p_{ij}c_{ij}, \quad i<j¡£
\end{equation}
and
\begin{equation}
\Pf_{q}(C)={\det}_q(T)\Pf_{q}(B).
\end{equation}
\end{theorem}
\begin{proof} We first check that $c_{ij}$ also form anti-symmetric matrix. We compute that
\begin{align*}
c_{ii}&=\sum_{k,l}t_{ki}b_{kl}t_{li}=\sum_{k<l}t_{ki}b_{kl}t_{li}+t_{li}b_{lk}t_{ki}\\
&=\sum_{k<l}(t_{ki}t_{li}-p_{kl}t_{li}t_{ki})b_{kl}=0.\\
\end{align*}
For $i<j$,
\begin{align*}
c_{ij}&=\sum_{k,l}t_{ki}b_{kl}t_{lj}
=\sum_{k<l}\left(t_{ki}b_{kl}t_{lj}+t_{li}b_{lk}t_{kj}\right)\\
&=\sum_{k<l}(t_{ki}t_{lj}-p_{kl}t_{li}t_{kj})b_{kl}
=\sum_{k<l}{\det}_q(T^{kl}_{ij})b_{kl},\\
\end{align*}
\begin{align*}
c_{ji}&=\sum_{k,l}t_{kj}b_{kl}t_{li}=\sum_{k<l}(t_{kj}t_{li}-p_{kl}t_{lj}t_{ki})b_{kl}\\
&=\sum_{k<l}-p_{ij}(t_{ki}t_{lj}-p_{kl}t_{li}t_{kj})b_{kl}
=-p_{ij}\sum_{k<l}{\det}_q(T^{kl}_{ij})b_{kl}\\
&=-p_{ij}c_{ij}.\\
\end{align*}

Consider the element
$$\Omega=X^{t}\dot{\ot} C \dot{\ot} X,$$
where we recall that $X=(x_1,\ldots,x_n)^t$ and $x_i\in \Lambda_q(x)$. Explicitly we have that
$\Omega=\sum_{1\leq i,j\leq n} c_{ij}x_{i}x_{j}=\sum_{i<j}(1+u^2) c_{ij}x_{i}x_{j}$, therefore
\begin{equation}\label{1}
 \Omega^n=(1+u^2)^{n}\Pf_q(C)x_{1}x_2\cdots x_{2n}.
\end{equation}

On the other hand, let $\omega_i=\sum_{j=1}^n t_{ij}\otimes x_j$.
Then
\begin{align*}
\omega_j\omega_i&=-q_{ij}\omega_i\omega_j,\quad i<j, \\
\omega_i\omega_i&=0.
\end{align*}
As $T \dot{\ot} X=(\omega_1,\ldots,\omega_{2n})^{t}$, one has that
\begin{align*}
\Omega&=X^{t} \dot{\ot} T^{t} B T \dot{\ot} X=(T \dot{\ot} X )^{t}B(T \dot{\ot} X )\\
&=\sum_{1\leq i,j\leq n}b_{ij}\omega_{i}\omega_{j}\\
&=\sum_{i<j}(1+u^2)b_{ij}\omega_{i}\omega_{j}.
\end{align*}
Therefore
\begin{equation}\label{2}
\begin{aligned}
& \Omega^n=(1+u^2)^{n}\Pf_q(B)\omega_1 \omega_2 \cdots \omega_{2n}\\
&=(1+u^2)^{n}\Pf_q(B){\det}_q(T)x_{1}x_2\cdots   x_{2n}
\end{aligned}
\end{equation}
Subsequently we have proved that
$$\Pf_{q}(C)={\det}_q(T)\Pf_{q}(B).$$
\end{proof}

The following column analog is clear.
\begin{remark}
Let $B$ be any matrix with entries $b_{ij},1\leq i,j \leq 2n$  commuting with $a_{ij}$ and $b_{ji}=-q_{ij}b_{ij},i<j$.
Let $C=TBT^{t}$. Then $c_{ji}=-q_{ij}c_{ij}, i<j$ and $\Pf_p(C)={\det}_q(T)\Pf_p(B)$.
\end{remark}

Let $b_{ji}=-p_{ij}b_{ij},i<j$, $b_{ji}'=-q_{ij}b_{ij}',i<j$. We define the quadratic elements $z_{ij}^{l}$ and $z_{ij}^{r}$, $1\leq i,j\leq n$ as follows:
\begin{equation}
z_{ij}^l=\sum_{k<l}b_{kl}  \xi^{kl}_{ij}, \ \
z_{ij}^r=\sum_{k<l}b_{kl}' \xi_{kl}^{ij}.
\end{equation}

Let $\mathcal{A}^l$(resp. $\mathcal{A}^r$ ) be the subalgebra generated by $z_{ij}^l$ ($z_{ij}^r$).
Since
$
\Delta(z_{ij}^l)=\sum_{s<t} z_{st}^l \ot  \xi^{st}_{ij}, \ \
\Delta(z_{ij}^r)=\sum_{s<t} \xi^{ij}_{st}  \ot  z_{st}^r,
$
$\mathcal{A}^l$(resp. $\mathcal{A}^r$ ) is left(resp. right) submodule of $ \GL_{p,u}(n) $ .
It is clear that $\Pf_{q}(z^l)$ (resp. $\Pf_{p}(z^r)$)  is annihilated by
all $l_{kl}^{\pm}$, $k\neq l$.
\begin{align}
&l_{k,k}^{+}.\Pf_{q}(z^l)=u^{2-n} \prod_{j<k} q_{jk} \prod_{j>k} p_{kj}  \Pf_{q}(z^l) ,\\
&l_{k,k}^{-}.\Pf_{q}(z^l)=u^{n-2} \prod_{j<k} p_{jk}^{-1} \prod_{j>k} q_{kj}^{-1}  \Pf_{q}(z^l),\\
&\Pf_{p}(z^r).l_{k,k}^{+}=u^{2-n} \prod_{j<k} q_{jk} \prod_{j>k} p_{kj}  \Pf_{p}(z^r)  ,\\
&\Pf_{p}(z^r).l_{k,k}^{-}=u^{n-2} \prod_{j<k} p_{jk}^{-1} \prod_{j>k} q_{kj}^{-1}  \Pf_{p}(z^r) .
\end{align}

\section{Multiparameter quantum hyper-Pfaffians}
We now generalize the notion of the quantum multiparameter Pfaffian to the quantum hyper-Pfaffian.
A hypermatrix $A=(A_{i_1\cdots i_n})$ is an array of entries indexed by several indices, while a matrix is
indexed by two indices.
\begin{definition}
Let $B$ be a hypermatrix with noncommutative entries $b_{i_1\cdots i_{m}},1\leq i_k\leq mn,k=1,\ldots, m.$
Multiparameter quantum hyper-Pfaffian is defined by
\begin{align*}
\Pf_q(B)
=\sum_{\sigma\in \Pi}(-q)_{\sigma}b_{\sigma(1)\cdots\sigma(m)}\cdots b_{\sigma(m(n-1)+1)\cdots\sigma(mn)},
\end{align*}
where $\Pi$ is the set of permutations $\sigma$ of $mn$ such that
$\sigma((k-1)m+1)<\sigma((k-1)m+2)<\sigma(km), k=1,\ldots,n.$
\end{definition}
Note that the multiparameter Pfaffian uses only the entries $b_{i_1\cdots i_m}$, where
$i_1<\cdots <i_m$.

Similar to Proposition {\ref{Pfaffian-Lap}}, one has the following result.
\begin{proposition}
For any $0\leq t\leq n$,
\begin{equation}
\Pf(B)=\sum_{I}inv(I,I^c)\Pf(B_{I})\Pf(B_{I^c}),
\end{equation}
where $I$ runs through subsets of $[1,mn]$ such that $|I|=mt$.
\end{proposition}
\begin{proof}

Let $\Omega=\sum_{i_1<\cdots <i_{m}}b_{i_1\cdots i_{m}}x_{i_1}x_{i_2}\cdots  x_{i_m}$,
then one has that
\begin{equation}\label{11}
 \Omega^n=\Pf_q(B)x_{1}x_2\cdots  x_{2n}.
\end{equation}
and
\begin{equation}\label{22}
 \Omega^n
=\Omega^t \Omega^{n-t}
=\sum_{I,J}\Pf(B_{I})\Pf(B_{I^c})x_Ix_{J},
\end{equation}
where as usual we have put $x_I=x_{i_1}x_{i_2}\cdots  x_{i_m}$.
Comparing \eqref{11} and \eqref{22}, one has the statement.
\end{proof}

\begin{theorem}
Let $B=(b_{i_1\cdots i_m })_{1\leq i_1,\ldots i_m \leq mn}$ be any hypermatrix with noncommutative entries commuting  
with those of the matrix $T=(t_{ij})$. Let
$$c_{I}=\sum_{J}{\det}_q(T^J_I) b_J,$$
then $\Pf_{q}(C)={\det}_q(T)\Pf_{q}(B)$.
\end{theorem}

\begin{proof}
Let $\delta_i=\sum_{j=1}^{mn}t_{ij}x_{j}$, and consider the element $\Omega=\sum c_{I}x_{I}$. It is clear that
\begin{equation}\label{hyperpf1}\Omega^n=\Pf_{q}(C)x_1x_2\cdots x_{2n}.
\end{equation}

On the other hand,
$\Omega=\sum b_{J}\delta_{J}$. Then
\begin{equation}\label{hyperpf2}
\Omega^n=\Pf(B)\delta_1\delta_2\cdots \delta_{2n}=\Pf(B){\det}_q(T)x_1x_2\cdots  x_{2n}.
\end{equation}

Comparing (\ref{hyperpf1}) and (\ref{hyperpf2}) we conclude that
$$\Pf_{q}(C)={\det}_q(T)\Pf_{q}(B).$$
\end{proof}

\begin{remark}
The column-analog is also true. In fact, one has the following result.
Let $B=(b_{i_1\cdots i_m })_{1\leq i_1,\ldots i_m \leq mn}$ be any hypermatrix with noncommutative entries commuting
with those of the matrix $T=(t_{ij})$. Let
$$c_{I}=\sum_{J}{\det}_q(T^I_J) b_J,$$
then $\Pf_{p}(C)={\det}_q(T)\Pf_{p}(B)$.

\end{remark}

\bigskip
\centerline{\bf Acknowledgments}
\medskip
The work is supported by National Natural Science Foundation of China (
11531004), Fapesp (2015/05927-0), Humboldt Foundation and Simons Foundation 523868.
Jing acknowledges the support of
Max-Planck Institute for Mathematics in the Sciences, Leipzig.
Both authors also thank South China University of Technology for support during the work.

\bibliographystyle{amsalpha}

\end{document}